\documentclass{amsart}
\usepackage{setspace}

\usepackage{amsthm}
\usepackage{a4}
\usepackage{amsfonts}
\usepackage{amsmath}
\usepackage{amssymb}
\usepackage{mathrsfs}
\usepackage{mathtools}

\usepackage{enumerate}

\usepackage{color}

\parskip\medskipamount
\DeclareFontFamily{OT1}{manual}{}
\DeclareFontShape{OT1}{manual}{m}{n}{ <10> manfnt }{}

\makeatletter
\def\blfootnote{\xdef\@thefnmark{}\@footnotetext}
\makeatother

\newcommand*{\<}{\langle}
\renewcommand*{\>}{\rangle}
\newcommand*{\x}{\times}

\newcommand{\leqs}{\leqslant}
\newcommand{\geqs}{\geqslant}

\newtheorem{lemma}{Lemma}[section]
\newtheorem{theorem}[lemma]{Theorem}
\newtheorem{prop}[lemma]{Proposition}
\newtheorem{cor}[lemma]{Corollary}

\newtheorem{question}[lemma]{Question}

\theoremstyle{definition}

\theoremstyle{remark}
\newtheorem{remark}[lemma]{Remark}

\theoremstyle{definition}

\title{Hyperbolicity over function fields of quadrics}
\begin{document}

\maketitle

%


\begin{abstract}\noindent A strong consequence of quadratic forms becoming hyperbolic over the function field of a form is established. This result is invoked to obtain a new characterisation of hyperbolicity over function fields, and to recover a number of important established results in quadratic form theory by means of novel, short and elementary proofs.
\end{abstract}

%






\blfootnote{James O'Shea, \emph{E-mail}: james.oshea@mu.ie\\ Department of Mathematics and Statistics, National University of Ireland Maynooth.}


\section{Introduction}

Viewing a quadratic form $p$ as a homogeneous polynomial of degree two, one is led to consider the algebraic variety $p=0$ and, in the case where $p$ is irreducible, the associated function field $F(p)$. These function fields were first profitably exploited in the early seventies, with a number of fundamental results in quadratic form theory being established in works by Arason and Pfister \cite{AP}, Elman and Lam \cite{EL}, Wadsworth \cite{W}, Knebusch \cite{Kn1}, \cite{Kn2} and others, and they have occupied a central role in the theory ever since, underpinning many of the acknowledged highlights in the field.








One of the main questions concerning function fields is to determine the Witt kernel of these extensions (see \cite[Question X.4.4A]{LAM}). Thus, one seeks to determine the quadratic forms $q$ over $F$ that become hyperbolic after scalar extension to $F(p)$. For $p$ of dimension $n$, an important necessary condition on such forms $q$ states that $p(x_1,\ldots ,x_n)$ must be a similarity factor of $q$ over the rational function field $F(x_1,\ldots ,x_n)$. This result, due to Arason and Pfister, plays a fundamental role in \cite{AP} and \cite{EL}. More broadly, Lam cites this result, together with an associated application of the Third Representation Theorem of Cassels and Pfister, as being a centrepiece result on function fields of quadratic forms (see \cite[p. 354]{LAM}). In \cite{Kn73}, Knebusch gave another proof of Arason and Pfister's result, and established that this necessary condition is also sufficient, thereby characterising those forms $q$ that become hyperbolic over $F(p)$.

In Corollary~\ref{EKMgen}, we establish a strong, general result on the consequences of a form $q$ becoming hyperbolic over a function field $F(p)$. Per Remark~\ref{recover}, the aforementioned theorem of Arason and Pfister readily follows from our result. Corollary~\ref{EKMgen} can be regarded as an extension of \cite[Lemma 3.1]{F1}, a result of Fitzgerald (indeed, we refer the reader to Fitzgerald's works \cite{F1} and \cite{F2} for a number of important results on Witt kernels, including a strong characterisation of scalar multiples of anisotropic Pfister forms in terms of hyperbolicity over function fields \cite[Theorem $1.6$]{F1} and a determination of the generators of the Witt kernels of $F(q)/F$ for $q$ of small dimension \cite[\S 2]{F2}). We prove our extension of Fitzgerald's result by considering the properties of multiples of Pfister forms, which, per Proposition~\ref{hypmultiples}, is a natural approach towards addressing the hyperbolicity question. In particular, we invoke Theorem~\ref{i1}, an important, elementary result on the isotropy of multiples of Pfister forms whose application, we suggest, may prove fruitful in other contexts within the literature. Consequently, we obtain new proofs of Fitzgerald's results \cite[Lemma 3.1]{F1} and \cite[Theorem 3.2]{F1}. In Theorem~\ref{hypchar}, we establish a new characterisation of hyperbolicity over function fields which incorporates Corollary~\ref{EKMgen}, and thus may be regarded as our main result. We further illustrate the fundamental nature of Corollary~\ref{EKMgen} by employing it to recover a number of important established results on quadratic forms. We invoke the Third Representation Theorem of Cassels and Pfister as a key tool in this regard, employing a result of Izhboldin to record a function-field characterisation of the subform property.

We let $F$ denote a field of characteristic different from two, and recall that every non-degenerate quadratic form on a vector space over $F$ can be diagonalised. We write $\<a_1,\ldots,a_n\>$ to denote the (n-dimensional) quadratic form on an $n$-dimensional $F$-vector space defined by $a_1,\ldots ,a_n\in F^{\x}$. We use the term ``form'' to refer to a non-degenerate quadratic form of positive dimension. If $p$ and $q$ are forms over $F$, we denote by $p\perp q$ their orthogonal sum and by $p\otimes q$ their tensor product. We use $aq$ to denote $\<a\>\otimes q$ for $a\in F^{\x}$. We write $p\simeq q$ to indicate that $p$ and $q$ are isometric, and say that $p$ and $q$ are \emph{similar} if $p\simeq aq$ for some $a\in F^{\x}$. For $q$ a form over $F$ and $K/F$ a field extension, we will employ the notation $q_K$ when viewing $q$ as a form over $K$ via the canonical embedding. A form $p$ is a \emph{subform of $q$} if $q\simeq p\perp r$ for some form $r$, in which case we write $p\subseteq q$. A form $q$ represents $a\in F$ if there exists a vector $v$ such that $q(v)=a$. We denote by $D_F(q)$ the set of values in $F^{\x}$ represented by $q$. A form over $F$ is \emph{isotropic} if it represents zero non-trivially, and \emph{anisotropic} otherwise. Every form $q$ has a decomposition $q\simeq q_{\mathrm{an}}\perp i(q)\x\<1,-1\>$ where the anisotropic quadratic form $q_{\mathrm{an}}$ and the non-negative integer $i(q)$ are uniquely determined. Thus, every $2$-dimensional isotropic form is isometric to $\<1,-1\>$. A form $q$ is \emph{hyperbolic} if $i(q)=\frac 1 2 \dim q$.









The similarity factors of $q$ constitute the group $G_F(q)=\{a\in F^{\x}\mid aq\simeq q\}$. Equivalently, $G_F(q)=\{a\in F^{\times}\mid \< 1,-a\>\otimes q\text{ is hyperbolic}\}$. A form $q$ over $F$ is said to be \emph{round} if $D_F(q)=G_F(q)$. We use $H_F(q)$ as notation for the set $\{a\in F^{\times}\mid \< 1,-a\>\otimes q\text{ is isotropic over }F\}$, which can be shown to coincide with $D_F(q)D_F(q)$, the set of products of two elements of $D_F(q)$ (see \cite[Lemme 2.5.4]{R}). Clearly we have that $G_F(q)\subseteq H_F(q)$ for all forms $q$ over $F$. If $1\in D_F(q)$, then we have that $G_F(q)\subseteq D_F(q)\subseteq H_F(q)$.

We will use Laurent series fields to represent ``generic'' extensions of our base field. We recall that every non-zero square class in $F(\!(x)\!)$, the Laurent series field in the variable $x$ over $F$, can be represented by $a$ or $ax$ for some $a\in F^{\x}$, whereby every form over $F(\!(x)\!)$ can be written as $p\perp xq$ for $p$ and $q$ forms over $F$. We will invoke the following result regarding the isotropy of forms over Laurent series fields.

\begin{lemma}\label{Hlemma} Let $p$ and $q$ be forms over $F$. Considering $p\perp x q$ as a form over $F(\!(x)\!)$, we have that $i(p\perp x q)=i(p)+ i(q)$.
\end{lemma}




For a form $p$ over $F$ with $\dim p=n\geqs 2$ and $p\not\simeq\<1,-1\>$, the \emph{function field $F(p)$} is the quotient field of the integral domain $F[X_1,\ldots ,X_n]/(p(X_1,\ldots ,X_n))$ (this is the function field of the affine quadric $p(X)=0$ over $F$). To avoid case distinctions, we set $F(p)=F$ if $\dim p= 1$ or $p\simeq\<1,-1\>$. For $p\not\simeq\<1,-1\>$, the field $F(p)$ is a purely-transcendental extension of $F$ if and only if $p$ is isotropic over $F$ (see \cite[Theorem X.$4.1$]{LAM}). We note the inclusion $F(\!(x)\!)(p)\subseteq F(p)(\!(x)\!)$, which we will apply in combination with Lemma~\ref{Hlemma}. For all forms $q$ over $F$ and all extensions $K/F$ such that $p_K$ is isotropic, we have that $i\left(q_{F(p)}\right)\leqs i(q_K)$ in accordance with Knebusch's specialisation results \cite[Proposition 3.1, Theorem 3.3]{Kn1}. In particular, we will invoke the following consequence of Knebusch's specialisation results:

\begin{lemma}\label{Knspec} Let $p$ be an anisotropic form over $F$ that contains a subform $r$ of dimension at least two. For $q$ a form over $F$, we have that $i\left(q_{F(r)}\right)\geqs i\left(q_{F(p)}\right)$.
\end{lemma}




For $n\in\mathbb{N}$, an \emph{$n$-fold Pfister form} over $F$ is a form isometric to $\<1,a_1\>\otimes\ldots\otimes\<1,a_n\>$ for some $a_1,\ldots ,a_n\in F^{\times}$ (the form $\< 1\>$ is the $0$-fold Pfister form). A form $q$ is a Pfister form over $F$ if and only if it is round over every extension of $F$ \cite[Satz 5, Theorem 2]{P}. A form $\tau$ is a neighbour of a Pfister form $\pi$ if $\tau$ is similar to a subform of $\pi$ and $\dim{\tau}>\frac 1 2 \dim\pi$. As isotropic Pfister forms are hyperbolic \cite[Theorem X.1.7]{LAM}, it follows that $\tau$ is isotropic over $F(\pi)$ for $\tau$ a neighbour of a Pfister form $\pi$ over $F$ (see \cite[Exercise I.16]{LAM}). Hence, in keeping with Knebusch's specialisation results, for $\tau$ a neighbour of a Pfister form $\pi$ and $q$ a form over $F$, we have that $i\left(q_{F(\tau)}\right)=i\left(q_{F(\pi)}\right)$.

\section{Hyperbolicity over function fields}



As above, we let $F$ denote a field of characteristic different from two and consider non-degenerate quadratic forms over $F$ of positive dimension. For $p$ a form over $F$, we address the problem of determining the kernel of the functorial map from the Witt ring of $F$ to the Witt ring of $F(p)$.




\begin{question}\label{hyp} For $p$ a form over $F$, what forms $q$ over $F$ become hyperbolic over $F(p)$?
\end{question}

In keeping with our identifications that $F(p)=F$ if $\dim p= 1$ or $p\simeq\<1,-1\>$, we may restrict our consideration of Question~\ref{hyp} to the case where $\dim p\geqs 2$. For $p\not\simeq\<1,-1\>$ in this situation, one has that $F(p)$ is a purely-transcendental extension of $F$ if and only if $p$ is isotropic over $F$ (see \cite[Theorem X.$4.1$]{LAM}), whereby we may further assume that $p$ is anisotropic.




 


For $p$ anisotropic with $\dim p=2$, we have that $p\simeq b\< 1,-a\>$ for some $a,b\in F^{\times}$, whereby $F(p)$ is a purely-transcendental extension of the quadratic extension $K=F(\sqrt{a})$ per \cite[Example X.3.9]{LAM}. Thus, in this case, it follows that $q_{F(p)}$ is hyperbolic if and only if $q_K$ is hyperbolic. By considering the ``rational'' and ``irrational'' parts of the equation that describes the isotropy of a form over $K=F(\sqrt{a})$, one can recover the following results \cite[Theorems VII.$3.1$, VII.$3.2$]{LAM} due to Pfister \cite{Pf3}.




\begin{theorem}\label{quadratic} Let $q$ be an anisotropic form over $F$. For $K=F(\sqrt{a})$ a quadratic extension of $F$, one has that \begin{enumerate}[$(i)$]
\item $q_K$ is isotropic if and only if $q\simeq b\< 1,-a\>\perp q'$ for $q'$ a form over $F$, $b\in F^{\times}$, 
\item $q_K$ is hyperbolic if and only if $q\simeq \< 1,-a\>\otimes p$ for some form $p$ over $F$.
\end{enumerate}
\end{theorem}


It is well known that the phenomenon described in Statement $(ii)$ above holds, more generally, for $K$ the function field of an $n$-fold Pfister form (this result is re-proven here as Corollary~\ref{pikernel}). With regard to Question~\ref{hyp}, we note that a description of the forms $q$ follows as a corollary of Theorem~\ref{quadratic}~$(ii)$.







\begin{prop}\label{hypmultiples} An anisotropic form $q$ over $F$ is hyperbolic over $F(p)$ for some form $p$ over $F$ if and only if $q\simeq \pi\otimes r$ for some $n$-fold Pfister form $\pi$ and form $r$ over $F$.\end{prop}


\begin{proof} If $q\simeq \pi\otimes r$ for $\pi$ an $n$-fold Pfister form and $r$ a form over $F$, then $q_{F(\pi)}$ is hyperbolic in accordance with \cite[Theorem X.1.7]{LAM}. Assuming that $q$ is hyperbolic over $F(p)$, it follows, as above, that $p$ is an anisotropic form over $F$ of dimension at least two. For $\beta$ a $2$-dimensional subform of $p$ over $F$, it follows that $q$ is hyperbolic over $F(\beta)$ in accordance with Lemma~\ref{Knspec}. Invoking Theorem~\ref{quadratic}~$(ii)$, it follows that $q\simeq \beta\otimes r$ for some form $r$ over $F$, establishing the result.
\end{proof}







Regarding Question~\ref{hyp}, one may argue as in the above proof to conclude that $q$ is a multiple of a $2$-fold in the case where $\dim p > 2$, as $3$-dimensional forms are Pfister neighbours. More generally, for $2^k < \dim p\leqs 2^{k+1}$, one can ask whether there exists a $(k+1)$-fold Pfister form $\pi$ over $F$ such that $p_{F(\pi)}$ is isotropic, whereby it would follow that $q$ is a multiple of $\pi$. We refer the reader to Scully's recent work \cite{Scully2} for striking results in this direction.






Given Proposition~\ref{hypmultiples}, one is naturally led to consider the properties of products of forms, and in particular those of multiples of Pfister forms, in connection with Question~\ref{hyp}. In \cite[Proposition $3.1$]{B}, Becher employed elementary properties of quadratic spaces to obtain the following result.

\begin{theorem}\label{kjb} Let $p$ and $q$ be forms over $F$ with $p$ anisotropic. If $p\otimes q$ is isotropic, then $p$ contains a subform $r$ such that $r\otimes q$ is isotropic and $\dim r\leqs \dim q$.
\end{theorem}



We remark that the above result holds for arbitrary forms $p$ in the case where $q$ is a form of dimension at least two. Theorem~\ref{kjb} generalises \cite[Proposition 2.2]{EL2}, Elman and Lam's important result on the isotropy of multiples of $2$-dimensional forms, which is recalled below.



\begin{cor}\label{EL73} Let $p$ and $q$ be forms over $F$ with $\dim p\geqs 2$ and $\dim q =2$. If $p\otimes q$ is isotropic, then $p$ contains a subform $r$ such that $r\otimes q$ is hyperbolic.
\end{cor}

\begin{proof} We may assume, without loss of generality, that $q$ is anisotropic. Invoking Theorem~\ref{kjb}, there exists a form $r\subseteq p$ of dimension at most two such that $r\otimes q$ is isotropic. As $q$ is anisotropic, it follows that $\dim r=2$. As the $4$-dimensional isotropic form $r\otimes q$ is of determinant one, it follows that $r\otimes q$ is hyperbolic.
\end{proof}




Pfister's Theorem~\ref{quadratic}~$(i)$, Elman and Lam's Corollary~\ref{EL73} and Knebusch's Lemma~\ref{Knspec} constitute the ingredients of a proof of the following result on the isotropy of Pfister multiples. This proof (not due to the author) can be found in \cite{OS2} (see Lemma $2.8$ and Corollary $2.9$), but it was known previously. 






\begin{theorem}\label{i1} Let $p$, $q$ and a Pfister form $\pi$ be forms over $F$. Let $K/F$ be any extension such that $\pi\otimes p$ is isotropic over $K$. Then $$i\left((\pi\otimes q)_{K}\right)\geqs (\dim \pi)i\left(q_{F(p)}\right).$$
\end{theorem}


With regard to Question~\ref{hyp}, the following result of Fitzgerald \cite[Lemma $3.1$]{F1} places a necessary condition on the forms $q$ (for a given form $p$). We have formulated this result in terms of the set $H_F(q)=\{a\in F^{\times}\mid \< 1,-a\>\otimes q\text{ is isotropic over }F\}$, recalling that $H_F(q)=D_F(q)D_F(q)$. 




\begin{theorem}\label{F31} Let $p$ and $q$ be forms over $F$ with $p$ anisotropic. If $q$ is hyperbolic over $F(p)$, then $H_F(p)\subseteq G_F(q)$. 
\end{theorem}






In \cite[Corollary $22.6$]{EKM}, the above result is stated with respect to anisotropic forms $p$ and $q$. We note that the result holds without any assumptions regarding anisotropy: if $p$ is isotropic, then $F(p)$ is a purely-transcendental extension of $F$, whereby $q$ is hyperbolic over $F$ by supposition, and thus $G_F(q)=F^{\x}$; whereas if $q$ is isotropic and not hyperbolic, we may restrict our attention to $q_{\text{an}}$, without loss of generality, as $G_F(q)=G_F\left(q_{\text{an}}\right)$. Thus, having removed the restrictions regarding anisotropy, the result can be extended to apply to all extensions $K/F$.

Fitzgerald uses \cite[Theorem $1.6$]{F1}, a powerful characterisation of scalar multiples of anisotropic Pfister forms, in combination with the Cassels-Pfister subform theorem and Knebusch's specialization results, to prove Theorem~\ref{F31}. In \cite{EKM}, the result is recovered as a corollary of a number of classical results on the behaviour of forms over rational function fields. By contrast, we prove the aforementioned generalisation of Theorem~\ref{F31} as an immediate corollary of Theorem~\ref{i1}.















\begin{cor}\label{EKMgen} Let $p$ and $q$ be forms over $F$. If $q$ is hyperbolic over $F(p)$, then $H_K(p)\subseteq G_K(q)$ for all extensions $K/F$.
\end{cor}

\begin{proof} Let $a\in H_K(p)$ for $K/F$. As $p\perp -ap$ is isotropic over $K$, it follows that $$i\left((q\perp -aq)_K\right)\geqs 2i\left(q_{F(p)}\right),$$ in accordance with Theorem~\ref{i1}. Hence, $q\perp -aq$ is hyperbolic over $K$, whereby $a\in G_K(q)$ and thus $H_K(p)\subseteq G_K(q)$.
\end{proof}





We can also recover Fitzgerald's \cite[Theorem $3.2$]{F1} as a corollary of the above result.

\begin{cor}\label{F32} Let $p$, $q$ and $\pi$ be forms over $F$, with $\pi$ an $n$-fold Pfister form. If $q$ is hyperbolic over $F(p)$, then $\pi\otimes q$ is hyperbolic over $F(\pi\otimes p)$.
\end{cor}

\begin{proof} We begin with the case where $\pi\simeq \< 1,-a\>$ for some $a\in F^{\times}$. As $a\in H_K(p)$ for $K=F(\pi\otimes p)$, an invocation of Corollary~\ref{EKMgen} yields that $a\in G_K(q)$, whereby $\pi\otimes q$ is hyperbolic over $F(\pi\otimes p)$. The general statement follows by iterating this argument.
\end{proof}




We may also employ Corollary~\ref{EKMgen} to recover the following characterisation of scalar multiples of Pfister forms due to Wadsworth \cite[Theorem 5]{W} and, independently, Knebusch \cite[Theorem 5.8]{Kn1} (see also \cite[Corollary $23.4$]{EKM}).



\begin{cor}\label{pfisterhypchar} An anisotropic form $q$ over $F$ of dimension at least two is a scalar multiple of a Pfister form if and only if $q$ is hyperbolic over $F(q)$.
\end{cor}


\begin{proof} As isotropic Pfister forms are hyperbolic \cite[Theorem X.1.7]{LAM}, it suffices to prove the ``if'' statement. If $q$ is hyperbolic over $F(q)$, it follows that $H_K(q)\subseteq G_K(q)$ for all extensions $K/F$ by Corollary~\ref{EKMgen}. As $1\in D_F(aq)$ for $a\in D_F(q)$, it follows that $aq$ is round over $K$ for all extensions $K/F$. Invoking Pfister's characterisation of such forms, \cite[Satz 5, Theorem 2]{P}, it follows that $aq$ is a Pfister form.
\end{proof}



We next show that the converse of Corollary~\ref{EKMgen} holds. To do this, we invoke \cite[Theorem $3.3$]{OS2}, a result on the isotropy of multiples of generic Pfister forms over iterated Laurent series fields. We offer a slight generalisation of this result below.


\begin{theorem}\label{itrans} Let $p$ and $q$ be forms over $F$ of dimension at least two. Over $K=F(\!(x_1)\!)\ldots (\!(x_n)\!)$, let $\pi\simeq \< 1,-x_1\>\otimes\ldots\otimes\< 1,-x_n\>$ and consider the forms $\pi\otimes p$ and $\pi\otimes q$. We have that $i\left((\pi\otimes q)_{K(\pi\otimes p)}\right)=(\dim\pi)i\left(q_{F(p)}\right)$.
\end{theorem}

\begin{proof} For $p$ and $q$ anisotropic forms over $F$, the proof of \cite[Theorem $3.3$]{OS2} can be applied verbatim to establish the result. Moreover, the result holds if $q$ is isotropic over $F$, as it suffices to establish the statement with respect to $q_{\text{an}}$ in this case. Finally, if $p$ is isotropic over $F$, we note that $F(p)=F$ in the case where  $p\simeq\< 1,-1\>$ and that $F(p)/F$ is a purely-transcendental extension otherwise. As $K(\pi\otimes p)$ is a purely-transcendental extension of $K$ in all cases, the result follows by invoking Lemma~\ref{Hlemma}.
\end{proof}




Combining Corollary~\ref{EKMgen} with Theorem~\ref{itrans}, we can establish the following characterisation of hyperbolicity over function fields.




 
\begin{theorem}\label{hypchar} Let $p$ and $q$ be forms over $F$ of dimension at least two. The following are equivalent:\begin{enumerate}[$(i)$]
\item $q$ is hyperbolic over $F(p)$,
\smallskip
\item $H_K(p)\subseteq G_K(q)$ for all extensions $K/F$,
\smallskip
\item $q\perp -xq$ is hyperbolic over $F(\!(x)\!)(p\perp -xp)$.
\end{enumerate}
Furthermore, in the case where $1\in D_F(p)$, the preceding statements are equivalent to the following: 
\begin{enumerate}[$(i)$]
\setcounter{enumi}{3}
\item $D_K(p)\subseteq G_K(q)$ for all extensions $K/F$,
\smallskip
\item $q\perp -xq$ is hyperbolic over $F(\!(x)\!)(p\perp -\< x\>)$.
\end{enumerate}
\end{theorem}


\begin{proof} We first establish the equivalence of the first three statements. Corollary~\ref{EKMgen} establishes that $(i)$ implies $(ii)$, while $(ii)$ clearly implies $(iii)$. Invoking Theorem~\ref{itrans} in the case where $n=1$, one sees that $(iii)$ implies $(i)$. 

In the case where $1\in D_F(p)$, we recall that $D_K(p)\subseteq H_K(p)$ for all extensions $K/F$, whereby $(ii)$ implies $(iv)$. As $(iv)$ clearly implies $(v)$, we conclude by showing that $(v)$ implies $(i)$. Invoking Lemma~\ref{Knspec}, it follows that $q\perp -xq$ is hyperbolic over $F(\!(x)\!)(p)$. Invoking the inclusion $F(\!(x)\!)(p)\subset F(p)(\!(x)\!)$, it follows that $q\perp -xq$ is hyperbolic over $F(p)(\!(x)\!)$, whereby an invocation of Lemma~\ref{Hlemma} establishes $(i)$.
\end{proof}


\begin{remark}\label{Roussey} We remark that Roussey established a related characterisation of isotropy over function fields (see \cite[Th\'eor\`eme 5.1.5]{R}). In particular, for $p$ and $q$ forms over $F$ of dimension at least two, one has that $q$ is isotropic over $F(p)$ if and only if $H_K(p)\subseteq H_K(q)$ for all extensions $K/F$. This result complements an existing characterisation due to Witt \cite{WITT}, Knebusch \cite{Kn73} and Bayer-Fluckiger (see \cite[Th\'eor\`eme 9]{BF}).
\end{remark}

We recall that a characterisation of hyperbolicity over function fields already exists in the literature. As highlighted in \cite[Th\'eor\`eme 9]{BF}, the following special case of Knebusch's result \cite[Theorem $4.2$]{Kn73} offers such a characterisation. Per the introduction, the ``only if'' statement had previously been employed by Arason and Pfister to prove the Hauptsatz in \cite{AP}.


\begin{theorem}\label{Knhyp} Let $p$ and $q$ be anisotropic forms over $F$ of dimension at least two. Suppose that $1\in D_F(p)$ and, for $\dim p=n$, let $K=F(x_1,\ldots ,x_n)$. Then $q$ is hyperbolic over $F(p)$ if and only if $p(x_1,\ldots ,x_n)\in G_K(q)$.
\end{theorem}

\begin{remark}\label{recover} We note that the ``only if'' statement in the above result can be recovered from Theorem~\ref{hypchar}, by invoking the fact that Statement $(i)$ implies Statement $(iv)$. The equivalence of Statements $(i)$ and $(v)$ in Theorem~\ref{hypchar} might reasonably be regarded as the function-field analogue of Theorem~\ref{Knhyp}.
\end{remark}

The equivalence of the first three statements in the following result constitutes the Third Representation Theorem of Cassels and Pfister \cite{P} (it is also known as the Subform Theorem). We next establish that the subform property can also be characterised using function fields.




\begin{theorem}\label{rep}
Let $p$ and $q$ be anisotropic quadratic forms over $F$ with $\dim p=n$ for $n\in\mathbb{N}$. The following are equivalent:
\begin{enumerate}[$(i)$]
\item $p\subseteq q$,
\smallskip
\item $D_K\left(p\right) \subseteq D_K\left(q\right)$ for every extension $K/F$, 
\smallskip
\item $p(x_1,\ldots ,x_n)\in D_{F(x_1,\ldots ,x_n)}\left(q\right)$.
\smallskip
\item $q\perp\< -x\>$ is isotropic over $F(\!(x)\!)(p\perp \< -x\>)$.
\end{enumerate}
\end{theorem}

\begin{proof} Assuming the Third Representation Theorem, we must show that the first three equivalent statements are also equivalent to $(iv)$. As $(ii)$ clearly implies $(iv)$, it suffices to prove that $(iv)$ implies $(i)$. Invoking \cite[Lemma 5.4 $(4)$]{Ilow}, it follows that $p\perp \< -x\>$ is similar to a subform of $q\perp \< -x\>$ over $F(\!(x)\!)$. As the non-zero square classes of $F(\!(x)\!)$ can be represented by $a$ or $ax$ for some $a\in F^{\x}$, we may invoke Lemma~\ref{Hlemma} to conclude that $p$ is a subform of $q$ over $F$. 
\end{proof}

\begin{remark}\label{nonind} We note that the above proof that Statement $(iv)$ characterises the subform property is dependent on the Third Representation Theorem, as this result is invoked in Izhboldin's proof of \cite[Lemma 5.4 $(4)$]{Ilow}.
\end{remark}








Invoking the Third Representation Theorem of Cassels and Pfister (incorporated in the preceding result), we next observe that the Cassels-Pfister Subform Theorem follows as a direct corollary of Corollary~\ref{EKMgen}. This proof differs from the original proofs of the result, which were independently obtained by Wadsworth \cite[Theorem 2]{W} and Knebusch \cite[Lemma 4.5]{Kn1}, and the proofs of the theorem presented in \cite{K}, \cite{EKM}, \cite{LAM} and \cite{S}, all of which also rely on the Third Representation Theorem.







\begin{cor}\label{cpst} Let $p$ and $q$ be anisotropic forms over $F$ of dimension at least two. If $q$ is hyperbolic over $F(p)$, then $ap\subseteq bq$ for all $a\in D_F(p)$ and $b\in D_F(q)$. 
\end{cor}

\begin{proof} Let $a\in D_F(p)$ and $b\in D_F(q)$. Since $q$ is hyperbolic over $F(p)$, it follows that $bq$ is hyperbolic over $F(ap)$. Hence, by Corollary~\ref{EKMgen}, it follows that $H_K(ap)\subseteq G_K(bq)$ for all extensions $K/F$. As $1\in D_F(ap)\cap D_F(bq)$, it follows that $D_K(ap)\subseteq H_K(ap)$ and $G_K(bq)\subseteq D_K(bq)$. Thus, $D_K(ap)\subseteq D_K(bq)$ for all extensions $K/F$, whereby the result follows from Theorem~\ref{rep}.
\end{proof}

%




As is well known, the Cassels-Pfister Subform Theorem may be invoked to recover the following answer to Question~\ref{hyp} in the case where $p$ is similar to a Pfister form. This result, implicit in Arason and Pfister's work \cite{AP}, was first established by Elman and Lam \cite{EL}, and it holds, more generally, for $p$ a neighbour of a Pfister form.

\begin{cor}\label{pikernel} Let $p$ and $q$ be anisotropic forms over $F$, with $p$ a neighbour of a Pfister form $\pi$ over $F$. If $q$ is hyperbolic over $F(p)$, then $q\simeq \pi\otimes r$ for some form $r$ over $F$. 
\end{cor}

\begin{proof} As $q$ is hyperbolic over $F(p)$, it follows that $q$ is hyperbolic over $F(\pi)$ by Knebusch's specialisation results. Applying Corollary~\ref{cpst}, we have that $a\pi\subseteq bq$ for some $a,b\in F^{\times}$, whereby $q\simeq ab\pi\perp q'$ for some form $q'$ over $F$. As $\pi$ is hyperbolic over $F(\pi)$, it follows that $q'$ is hyperbolic over $F(\pi)$, whereby an induction argument establishes the result.
\end{proof}

Invoking the fact that $n$-fold Pfister forms additively generate $I^n(F)$, the $n$th power of the fundamental ideal of even-dimensional forms in the Witt ring $W(F)$, the Hauptsatz of Arason and Pfister \cite[Hauptsatz and Kor. 3]{AP}, recorded below, can be recovered as a corollary of Corollary~\ref{pikernel} and Corollary~\ref{pfisterhypchar}. We refer to \cite[p. 88]{Hdublin} for a short proof of this fact.



\begin{cor} Let $q$ be a form over $F$ such that $q\in I^n(F)$. If $\dim q< 2^n$, then $q$ is hyperbolic. If $\dim q=2^n$, then $q$ is similar to an $n$-fold Pfister form.\end{cor}

\begin{remark}\label{LAM} Lam uses Arason and Pfister's consequence of hyperbolicity over function fields (incorporated in Theorem~\ref{Knhyp}) together an associated application of the Third Representation Theorem of Cassels and Pfister as the cornerstone of his exposition of basic theorems on function fields. Thus, we refer the reader to \cite[\S X.4]{LAM} for other established results that may be recovered as corollaries of Corollary~\ref{EKMgen} and the Third Representation Theorem.
\end{remark}

\end{document}